\documentclass[12pt]{amsart}

\usepackage{amsmath}
\usepackage{amssymb}
\usepackage{enumerate}
\usepackage{graphicx}

\makeatletter
\@namedef{subjclassname@2010}{%
  \textup{2010} Mathematics Subject Classification}
\makeatother

\newtheorem{thm}{Theorem}[section]
\newtheorem{cor}[thm]{Corollary}
\newtheorem{lem}[thm]{Lemma}

\theoremstyle{definition}
\newtheorem{defin}[thm]{Definition}

\numberwithin{equation}{section}

\begin{document}


\baselineskip=17pt



\title[Morrey spaces related to nonnegative potentials]{Morrey spaces related to certain nonnegative potentials and fractional integrals on the Heisenberg groups}

\author[H. Wang]{Hua Wang}
\address{College of Mathematics and Econometrics, Hunan University, Changsha, 410082, P. R. China\\
\&~Department of Mathematics and Statistics, Memorial University, St. John's, NL A1C 5S7, Canada}
\email{wanghua@pku.edu.cn}
\date{}

\begin{abstract}
Let $\mathcal L=-\Delta_{\mathbb H^n}+V$ be a Schr\"odinger operator on the Heisenberg group $\mathbb H^n$, where $\Delta_{\mathbb H^n}$ is the sub-Laplacian on $\mathbb H^n$ and the nonnegative potential $V$ belongs to the reverse H\"older class $RH_s$ with $s\geq Q/2$. Here $Q=2n+2$ is the homogeneous dimension of $\mathbb H^n$. For given $\alpha\in(0,Q)$, the fractional integrals associated to the Schr\"odinger operator $\mathcal L$ is defined by $\mathcal I_{\alpha}={\mathcal L}^{-{\alpha}/2}$. In this article, we first introduce the Morrey space $L^{p,\kappa}_{\rho,\infty}(\mathbb H^n)$ and weak Morrey space $WL^{p,\kappa}_{\rho,\infty}(\mathbb H^n)$ related to the nonnegative potential $V$. Then we establish the boundedness of fractional integrals ${\mathcal L}^{-{\alpha}/2}$ on these new spaces. Furthermore, in order to deal with certain extreme cases, we also introduce the spaces $\mathrm{BMO}_{\rho,\infty}(\mathbb H^n)$ and $\mathcal{C}^{\beta}_{\rho,\infty}(\mathbb H^n)$ with exponent $\beta\in(0,1]$.
\end{abstract}

\maketitle

\section{Introduction}

\subsection{Heisenberg group $\mathbb H^n$}
The \emph{Heisenberg group} $\mathbb H^n$ is a nilpotent Lie group with underlying manifold $\mathbb C^n\times\mathbb R$. The group structure (the multiplication law) is given by
\begin{equation*}
(z,t)\cdot(z',t'):=\Big(z+z',t+t'+2\mathrm{Im}(z\cdot\overline{z'})\Big),
\end{equation*}
where $z=(z_1,z_2,\dots,z_n)$, $z'=(z_1',z_2',\dots,z_n')\in\mathbb C^n$, and
\begin{equation*}
z\cdot\overline{z'}:=\sum_{j=1}^nz_j\overline{z_j'}.
\end{equation*}
It can be easily seen that the inverse element of $u=(z,t)$ is $u^{-1}=(-z,-t)$, and the identity is the origin $(0,0)$. The Lie algebra of left-invariant vector fields on $\mathbb H^n$ is spanned by
\begin{equation*}
\begin{cases}
X_j=\displaystyle\frac{\partial}{\partial x_j}+2y_j\frac{\partial}{\partial t},\quad j=1,2,\dots,n,&\\
Y_j=\displaystyle\frac{\partial}{\partial y_j}-2x_j\frac{\partial}{\partial t},\quad j=1,2,\dots,n,&\\
T=\displaystyle\frac{\partial}{\partial t}.&
\end{cases}
\end{equation*}
All non-trivial commutation relations are given by
\begin{equation*}
[X_j,Y_j]=-4T,\quad j=1,2,\dots,n.
\end{equation*}
The sub-Laplacian $\Delta_{\mathbb H^n}$ is defined by
\begin{equation*}
\Delta_{\mathbb H^n}:=\sum_{j=1}^n\big(X_j^2+Y_j^2\big).
\end{equation*}
The dilations on $\mathbb H^n$ have the following form
\begin{equation*}
\delta_a(z,t):=(az,a^2t),\quad a>0.
\end{equation*}
For given $(z,t)\in\mathbb H^n$, the \emph{homogeneous norm} of $(z,t)$ is given by
\begin{equation*}
|(z,t)|=\big(|z|^4+t^2\big)^{1/4}.
\end{equation*}
Observe that $|(z,t)^{-1}|=|(z,t)|$ and
\begin{equation*}
\big|\delta_a(z,t)\big|=\big(|az|^4+(a^2t)^2\big)^{1/4}=a|(z,t)|.
\end{equation*}
In addition, this norm $|\cdot|$ satisfies the triangle inequality and leads to a left-invariant distant $d(u,v)=\big|u^{-1}\cdot v\big|$ for $u=(z,t)$, $v=(z',t')\in\mathbb H^n$. The ball of radius $r$ centered at $u$ is denoted by
\begin{equation*}
B(u,r)=\big\{v\in\mathbb H^n:d(u,v)<r\big\}.
\end{equation*}
The Haar measure on $\mathbb H^n$ coincides with the Lebesgue measure on $\mathbb R^{2n}\times\mathbb R$. The measure of any measurable set $E\subset\mathbb H^n$ is denoted by $|E|$. For $(u,r)\in\mathbb H^n\times(0,\infty)$, it can be shown that the volume of $B(u,r)$ is
\begin{equation*}
|B(u,r)|=r^{Q}\cdot|B(0,1)|,
\end{equation*}
where $Q:=2n+2$ is the \emph{homogeneous dimension} of $\mathbb H^n$ and $|B(0,1)|$ is the volume of the unit ball in $\mathbb H^n$. A direct calculation shows that
\begin{equation*}
|B(0,1)|=\frac{2\pi^{n+\frac{\,1\,}{2}}\Gamma(\frac{\,n\,}{2})}{(n+1)\Gamma(n)\Gamma(\frac{n+1}{2})}.
\end{equation*}
Given a ball $B=B(u,r)$ in $\mathbb H^n$ and $\lambda>0$, we shall use the notation $\lambda B$ to denote $B(u,\lambda r)$. Clearly, we have
\begin{equation}\label{homonorm}
|B(u,\lambda r)|=\lambda^{Q}\cdot|B(u,r)|.
\end{equation}
For more information about the harmonic analysis on the Heisenberg groups, we refer the reader to \cite[Chapter XII]{stein2} and \cite{thangavelu}.

Let $V:\mathbb H^n\rightarrow\mathbb R$ be a nonnegative locally integrable function that belongs to the \emph{reverse H\"older class} $RH_s$ for some exponent $1<s<\infty$; i.e., there exists a positive constant $C>0$ such that the following reverse H\"older inequality
\begin{equation*}
\left(\frac{1}{|B|}\int_B V(w)^s\,dw\right)^{1/s}\leq C\left(\frac{1}{|B|}\int_B V(w)\,dw\right)
\end{equation*}
holds for every ball $B$ in $\mathbb H^n$. For given $V\in RH_s$ with $s\geq Q/2$, we introduce the \emph{critical radius function} $\rho(u)=\rho(u;V)$ which is given by
\begin{equation}\label{rho}
\rho(u):=\sup\bigg\{r>0:\frac{1}{r^{Q-2}}\int_{B(u,r)}V(w)\,dw\leq1\bigg\},\quad u\in\mathbb H^n,
\end{equation}
where $B(u,r)$ denotes the ball in $\mathbb H^n$ centered at $u$ and with radius $r$. It is well known that this auxiliary function satisfies $0<\rho(u)<\infty$ for any $u\in\mathbb H^n$ under the above assumption on $V$ (see \cite{lu}). We need the following known result concerning the critical radius function \eqref{rho}.
\begin{lem}[\cite{lu}]\label{N0}
If $V\in RH_s$ with $s\geq Q/2$, then there exist constants $C_0\geq 1$ and $N_0>0$ such that for all $u$ and $v$ in $\mathbb H^n$,
\begin{equation}\label{com}
\frac{\,1\,}{C_0}\left(1+\frac{|v^{-1}u|}{\rho(u)}\right)^{-N_0}\leq\frac{\rho(v)}{\rho(u)}\leq C_0\left(1+\frac{|v^{-1}u|}{\rho(u)}\right)^{\frac{N_0}{N_0+1}}.
\end{equation}
\end{lem}
Lemma \ref{N0} is due to Lu \cite{lu}. In the setting of $\mathbb R^n$, this result was given by Shen in \cite{shen}. As a straightforward consequence of \eqref{com}, we can see that for each integer $k\geq1$, the following estimate
\begin{equation}\label{com2}
1+\frac{2^kr}{\rho(v)}\geq \frac{1}{C_0}\left(1+\frac{r}{\rho(u)}\right)^{-\frac{N_0}{N_0+1}}\left(1+\frac{2^kr}{\rho(u)}\right)
\end{equation}
holds for any $v\in B(u,r)$ with $u\in\mathbb H^n$ and $r>0$, $C_0$ is the same as in \eqref{com}.

\subsection{Fractional integrals}
First we recall the fractional power of the Laplacian operator on $\mathbb R^n$. For given $\alpha\in(0,n)$, the classical fractional integral operator $I^{\Delta}_{\alpha}$ (also referred to as the Riesz potential) is defined by
\begin{equation*}
I^{\Delta}_{\alpha}(f):=(-\Delta)^{-\alpha/2}(f),
\end{equation*}
where $\Delta$ is the Laplacian operator on $\mathbb R^n$. If $f\in\mathcal S(\mathbb R^n)$, then by virtue of the Fourier transform, we have
\begin{equation*}
\widehat{I^{\Delta}_{\alpha}f}(\xi)=(2\pi|\xi|)^{-\alpha}\widehat{f}(\xi),\quad \forall\,\xi\in\mathbb R^n.
\end{equation*}
Comparing this to the Fourier transform of $|x|^{-\alpha}$, $0<\alpha<n$, we are led to redefine the fractional integral operator $I^{\Delta}_{\alpha}$ by
\begin{equation}\label{frac}
I^{\Delta}_{\alpha}f(x):=\frac{1}{\gamma(\alpha)}\int_{\mathbb R^n}\frac{f(y)}{|x-y|^{n-\alpha}}\,dy,
\end{equation}
where
\begin{equation*}
\gamma(\alpha)=\frac{\pi^{\frac{n}{\,2\,}}2^\alpha\Gamma(\frac{\alpha}{\,2\,})}{\Gamma(\frac{n-\alpha}{2})}
\end{equation*}
with $\Gamma(\cdot)$ being the usual gamma function. It is well-known that the Hardy-Littlewood-Sobolev theorem states that the fractional integral operator $I^{\Delta}_{\alpha}$ is bounded from $L^p(\mathbb R^n)$ to $L^q(\mathbb R^n)$ for $0<\alpha<n$, $1<p<n/{\alpha}$ and $1/q=1/p-{\alpha}/n$. Also we know that $I^{\Delta}_{\alpha}$ is bounded from $L^1(\mathbb R^n)$ to $WL^q(\mathbb R^n)$ for $0<\alpha<n$ and $q=n/{(n-\alpha)}$ (see \cite{stein}).

Next we are going to discuss the fractional integrals on the Heisenberg group. For given $\alpha\in(0,Q)$ with $Q=2n+2$, the fractional integral operator $I_{\alpha}$ (also referred to as the Riesz potential) is defined by (see \cite{xiao})
\begin{equation}\label{frac2}
I_{\alpha}(f):=(-\Delta_{\mathbb H^n})^{-\alpha/2}(f),
\end{equation}
where $\Delta_{\mathbb H^n}$ is the sub-Laplacian on $\mathbb H^n$ defined above. Let $f$ and $g$ be integrable functions defined on $\mathbb H^n$. Define the \emph{convolution} $f*g$ by
\begin{equation*}
(f*g)(u):=\int_{\mathbb H^n}f(v)g(v^{-1}u)\,dv.
\end{equation*}
We denote by $H_s(u)$ the convolution kernel of heat semigroup $\big\{T_s=e^{s\Delta_{\mathbb H^n}}:s>0\big\}$. Namely,
\begin{equation*}
e^{s\Delta_{\mathbb H^n}}f(u)=\int_{\mathbb H^n}H_s(v^{-1}u)f(v)\,dv.
\end{equation*}
For any $u=(z,t)\in\mathbb H^n$, it was proved in \cite[Theorem 4.2]{xiao} that $I_{\alpha}$ can be expressed by the following formula:
\begin{equation}\label{frac3}
\begin{split}
I_{\alpha}f(u)&=\frac{1}{\Gamma(\alpha/2)}\int_0^{\infty}e^{s\Delta_{\mathbb H^n}}f(u)\,s^{\alpha/2-1}ds\\
&=\frac{1}{\Gamma(\alpha/2)}\int_0^{\infty}\big(H_s*f\big)(u)\,s^{\alpha/2-1}ds.
\end{split}
\end{equation}
Let $V\in RH_s$ for $s\geq Q/2$. For such a potential $V$, we consider the time independent \emph{Schr\"odinger operator} on $\mathbb H^n$ (see \cite{lin}),
\begin{equation*}
\mathcal L:=-\Delta_{\mathbb H^n}+V,
\end{equation*}
and its associated semigroup
\begin{equation*}
\mathcal T^{\mathcal L}_sf(u):=e^{-s\mathcal L}f(u)=\int_{\mathbb H^n}P_s(u,v)f(v)\,dv,\quad f\in L^2(\mathbb H^n),~s>0,
\end{equation*}
where $P_s(u,v)$ denotes the kernel of the operator $e^{-s\mathcal L},s>0$. For any $u=(z,t)\in\mathbb H^n$, it is well-known that the heat kernel $H_s(u)$ has the explicit expression:
\begin{equation*}
H_s(z,t)=(2\pi)^{-1}(4\pi)^{-n}\int_{\mathbb R}\bigg(\frac{|\lambda|}{\sinh|\lambda|s}\bigg)^n\exp\left\{-\frac{|\lambda||z|^2}{4}\coth|\lambda|s-i\lambda s\right\}d\lambda,
\end{equation*}
and hence it satisfies the following estimate (see \cite{jerison} for instance)
\begin{equation}\label{heatkernel}
0\leq H_s(u)\leq C\cdot s^{-Q/2}\exp\bigg(-\frac{|u|^2}{As}\bigg),
\end{equation}
where the constants $C,A>0$ are independent of $s$ and $u\in\mathbb H^n$. Since $V\geq0$, by the \emph{Trotter product formula} and \eqref{heatkernel}, one has
\begin{equation}\label{heat}
0\leq P_s(u,v)\leq H_s(v^{-1}u)\leq C\cdot s^{-Q/2}\exp\bigg(-\frac{|v^{-1}u|^2}{As}\bigg),\quad s>0.
\end{equation}
Moreover, this estimate \eqref{heat} can be improved when $V$ belongs to the reverse H\"older class $RH_s$ for some $s\geq Q/2$. The auxiliary function $\rho(u)$ arises naturally in this context.
\begin{lem}\label{ker1}
Let $V\in RH_s$ with $s\geq Q/2$, and let $\rho(u)$ be the auxiliary function determined by $V$. For every positive integer $N\geq1$, there exists a positive constant $C_N>0$ such that for all $u$ and $v$ in $\mathbb H^n$,
\begin{equation*}
0\leq P_s(u,v)\leq C_N\cdot s^{-Q/2}\exp\bigg(-\frac{|v^{-1}u|^2}{As}\bigg)\bigg(1+\frac{\sqrt{s\,}}{\rho(u)}+\frac{\sqrt{s\,}}{\rho(v)}\bigg)^{-N},\quad s>0.
\end{equation*}
\end{lem}
This estimate of $P_s(u,v)$ is better than \eqref{heat}, which was given by Lin and Liu in \cite[Lemma 7]{lin}.

Inspired by \eqref{frac2} and \eqref{frac3}, for given $\alpha\in(0,Q)$, the \emph{$\mathcal L$-Fractional integral operator} or \emph{$\mathcal L$-Riesz potential} on the Heisenberg group is defined by (see \cite{jiang} and \cite{jiang2})
\begin{equation*}
\begin{split}
\mathcal I_{\alpha}(f)(u)&:={\mathcal L}^{-{\alpha}/2}f(u)\\
&=\frac{1}{\Gamma(\alpha/2)}\int_0^{\infty}e^{-s\mathcal L}f(u)\,s^{\alpha/2-1}ds.
\end{split}
\end{equation*}
Recall that in the setting of $\mathbb R^n$, this integral operator was first introduced by Dziuba\'{n}ski et al.\cite{dziu}. In this article we shall be interested in the behavior of the fractional integral operator $\mathcal I_{\alpha}$ associated to Schr\"odinger operator on $\mathbb H^n$. For $1\leq p<\infty$, the Lebesgue space $L^p(\mathbb H^n)$ is defined to be the set of all measurable functions $f$ on $\mathbb H^n$ such that
\begin{equation*}
\big\|f\big\|_{L^p(\mathbb H^n)}:=\bigg(\int_{\mathbb H^n}\big|f(u)\big|^p\,du\bigg)^{1/p}<\infty.
\end{equation*}
The weak Lebesgue space $WL^p(\mathbb H^n)$ consists of all measurable functions $f$ on $\mathbb H^n$ such that
\begin{equation*}
\big\|f\big\|_{WL^p(\mathbb H^n)}:=
\sup_{\lambda>0}\lambda\cdot\big|\big\{u\in\mathbb H^n:|f(u)|>\lambda\big\}\big|^{1/p}<\infty.
\end{equation*}
Now we are going to establish strong-type and weak-type estimates of the $\mathcal L$-fractional integral operator $\mathcal I_{\alpha}$ on the Lebesgue spaces. We first claim that the following estimate
\begin{equation}\label{claim}
|\mathcal I_{\alpha}f(u)|\leq C\int_{\mathbb H^n}|f(v)|\frac{1}{|v^{-1}u|^{Q-\alpha}}\,dv=C\big(|f|*|\cdot|^{\alpha-Q}\big)(u)
\end{equation}
holds for all $u\in\mathbb H^n$. Let us verify \eqref{claim}. To do so, denote by $\mathcal K_{\alpha}(u,v)$ the kernel of the fractional integral operator $\mathcal I_{\alpha}$. Then we have
\begin{equation*}
\begin{split}
\int_{\mathbb H^n}\mathcal K_{\alpha}(u,v)f(v)\,dv&=\mathcal I_{\alpha}f(u)={\mathcal L}^{-{\alpha}/2}f(u)\\
&=\frac{1}{\Gamma(\alpha/2)}\int_0^{\infty}e^{-s\mathcal L}f(u)\,s^{\alpha/2-1}ds\\
&=\int_0^{\infty}\bigg[\frac{1}{\Gamma(\alpha/2)}\int_{\mathbb H^n}P_s(u,v)f(v)\,dv\bigg]s^{\alpha/2-1}ds\\
&=\int_{\mathbb H^n}\bigg[\frac{1}{\Gamma(\alpha/2)}\int_0^{\infty}P_s(u,v)\,s^{\alpha/2-1}ds\bigg]f(v)\,dv.
\end{split}
\end{equation*}
Hence,
\begin{equation*}
\mathcal K_{\alpha}(u,v)=\frac{1}{\Gamma(\alpha/2)}\int_0^{\infty}P_s(u,v)\,s^{\alpha/2-1}ds.
\end{equation*}
Moreover, by using \eqref{heat}, we can deduce that
\begin{equation*}
\begin{split}
\big|\mathcal K_{\alpha}(u,v)\big|&\leq\frac{C}{\Gamma(\alpha/2)}\int_0^{\infty}\exp\bigg(-\frac{|v^{-1}u|^2}{As}\bigg)s^{\alpha/2-Q/2-1}ds\\
&\leq\frac{C}{\Gamma(\alpha/2)}\cdot\frac{1}{|v^{-1}u|^{Q-\alpha}}\int_0^{\infty}e^{-t}\,t^{(Q/2-\alpha/2)-1}dt\\
&=C\cdot\frac{\Gamma(Q/2-\alpha/2)}{\Gamma(\alpha/2)}\cdot\frac{1}{|v^{-1}u|^{Q-\alpha}},
\end{split}
\end{equation*}
where in the second step we have used a change of variables. Thus \eqref{claim} holds. According to Theorems 4.4 and 4.5 in \cite{xiao}, we get the Hardy-Littlewood-Sobolev theorem on the Heisenberg group.
\begin{thm}\label{strong}
Let $0<\alpha<Q$ and $1\leq p<Q/{\alpha}$. Define $1<q<\infty$ by the relation $1/q=1/p-{\alpha}/Q$. Then the following statements are valid:
\begin{enumerate}
  \item if $p>1$, then $\mathcal I_{\alpha}$ is bounded from $L^p(\mathbb H^n)$ to $L^q(\mathbb H^n)$;
  \item if $p=1$, then $\mathcal I_{\alpha}$ is bounded from $L^1(\mathbb H^n)$ to $WL^q(\mathbb H^n)$.
\end{enumerate}
\end{thm}

The organization of this paper is as follows. In Section 2, we will give the definitions of Morrey space and weak Morrey space and state our main results: Theorems \ref{mainthm:1}, \ref{mainthm:2} and \ref{mainthm:3}. Section 3 is devoted to proving the boundedness of the fractional integral operator in the context of Morrey spaces. We will study certain extreme cases in Section 4. Throughout this paper, $C$ represents a positive constant that is independent of the main parameters, but may be different from line to line, and a subscript is added when we wish to make clear its dependence on the parameter in the subscript. We also use $a\approx b$ to denote the equivalence of $a$ and $b$; that is, there exist two positive constants $C_1$, $C_2$ independent of $a,b$ such that $C_1a\leq b\leq C_2a$.

\section{Main results}
In this section, we introduce some types of Morrey spaces related to the nonnegative potential $V$ on $\mathbb H^n$, and then give our main results.
\begin{defin}
Let $\rho$ be the auxiliary function determined by $V\in RH_s$ with $s\geq Q/2$. Let $1\leq p<\infty$ and $0\leq\kappa<1$. For given $0<\theta<\infty$, the Morrey space $L^{p,\kappa}_{\rho,\theta}(\mathbb H^n)$ is defined to be the set of all $p$-locally integrable functions $f$ on $\mathbb H^n$ such that
\begin{equation}\label{morrey1}
\bigg(\frac{1}{|B|^{\kappa}}\int_B\big|f(u)\big|^p\,du\bigg)^{1/p}
\leq C\cdot\left(1+\frac{r}{\rho(u_0)}\right)^{\theta}
\end{equation}
for every ball $B=B(u_0,r)$ in $\mathbb H^n$. A norm for $f\in L^{p,\kappa}_{\rho,\theta}(\mathbb H^n)$, denoted by $\|f\|_{L^{p,\kappa}_{\rho,\theta}(\mathbb H^n)}$, is given by the infimum of the constants in \eqref{morrey1}, or equivalently,
\begin{equation*}
\big\|f\big\|_{L^{p,\kappa}_{\rho,\theta}(\mathbb H^n)}:=\sup_{B(u_0,r)}\left(1+\frac{r}{\rho(u_0)}\right)^{-\theta}
\bigg(\frac{1}{|B|^{\kappa}}\int_B\big|f(u)\big|^p\,du\bigg)^{1/p}
<\infty,
\end{equation*}
where the supremum is taken over all balls $B=B(u_0,r)$ in $\mathbb H^n$, $u_0$ and $r$ denote the center and radius of $B$ respectively. Define
\begin{equation*}
L^{p,\kappa}_{\rho,\infty}(\mathbb H^n):=\bigcup_{\theta>0}L^{p,\kappa}_{\rho,\theta}(\mathbb H^n).
\end{equation*}
\end{defin}

\begin{defin}
Let $\rho$ be the auxiliary function determined by $V\in RH_s$ with $s\geq Q/2$. Let $1\leq p<\infty$ and $0\leq\kappa<1$. For given $0<\theta<\infty$, the weak Morrey space $WL^{p,\kappa}_{\rho,\theta}(\mathbb H^n)$ is defined to be the set of all measurable functions $f$ on $\mathbb H^n$ such that
\begin{equation*}
\frac{1}{|B|^{\kappa/p}}\sup_{\lambda>0}\lambda\cdot\big|\big\{u\in B:|f(u)|>\lambda\big\}\big|^{1/p}
\leq C\cdot\left(1+\frac{r}{\rho(u_0)}\right)^{\theta}
\end{equation*}
for every ball $B=B(u_0,r)$ in $\mathbb H^n$, or equivalently,
\begin{equation*}
\big\|f\big\|_{WL^{p,\kappa}_{\rho,\theta}(\mathbb H^n)}:=\sup_{B(u_0,r)}\left(1+\frac{r}{\rho(u_0)}\right)^{-\theta}\frac{1}{|B|^{\kappa/p}}
\sup_{\lambda>0}\lambda\cdot\big|\big\{u\in B:|f(u)|>\lambda\big\}\big|^{1/p}<\infty.
\end{equation*}
Correspondingly, we define
\begin{equation*}
WL^{p,\kappa}_{\rho,\infty}(\mathbb H^n):=\bigcup_{\theta>0}WL^{p,\kappa}_{\rho,\theta}(\mathbb H^n).
\end{equation*}
\end{defin}
Obviously, if we take $\theta=0$ or $V\equiv0$, then this Morrey space (or weak Morrey space) is just the Morrey space $L^{p,\kappa}(\mathbb H^n)$ (or $WL^{p,\kappa}(\mathbb H^n)$), which was defined by Guliyev et al.\cite{guliyev}. Moreover, according to the above definitions, one has
\begin{equation*}
\begin{cases}
L^{p,\kappa}(\mathbb H^n)\subset L^{p,\kappa}_{\rho,\theta_1}(\mathbb H^n)\subset L^{p,\kappa}_{\rho,\theta_2}(\mathbb H^n);&\\
WL^{p,\kappa}(\mathbb H^n)\subset WL^{p,\kappa}_{\rho,\theta_1}(\mathbb H^n)\subset WL^{p,\kappa}_{\rho,\theta_2}(\mathbb H^n),&
\end{cases}
\end{equation*}
for $0<\theta_1<\theta_2<\infty$. Hence $L^{p,\kappa}(\mathbb H^n)\subset L^{p,\kappa}_{\rho,\infty}(\mathbb H^n)$ and $WL^{p,\kappa}(\mathbb H^n)\subset WL^{p,\kappa}_{\rho,\infty}(\mathbb H^n)$ for $(p,\kappa)\in[1,\infty)\times[0,1)$. The space $L^{p,\kappa}_{\rho,\theta}(\mathbb H^n)$ (or $WL^{p,\kappa}_{\rho,\theta}(\mathbb H^n)$) could be viewed as an extension of Lebesgue (or weak Lebesgue) space on $\mathbb H^n$ (when $\kappa=\theta=0$). In this article we will extend the Hardy-Littlewood-Sobolev theorem on $\mathbb H^n$ to the Morrey spaces. We now present our main results as follows.

\begin{thm}\label{mainthm:1}
Let $0<\alpha<Q$, $1<p<Q/{\alpha}$ and $1/q=1/p-{\alpha}/Q$. If $V\in RH_s$ with $s\geq Q/2$ and $0<\kappa<p/q$, then the $\mathcal L$-fractional integral operator $\mathcal I_{\alpha}$ is bounded from $L^{p,\kappa}_{\rho,\infty}(\mathbb H^n)$ into $L^{q,{(\kappa q)}/p}_{\rho,\infty}(\mathbb H^n)$.
\end{thm}

\begin{thm}\label{mainthm:2}
Let $0<\alpha<Q$, $p=1$ and $q=Q/{(Q-\alpha)}$. If $V\in RH_s$ with $s\geq Q/2$ and $0<\kappa<1/q$, then the $\mathcal L$-fractional integral operator $\mathcal I_{\alpha}$ is bounded from $L^{1,\kappa}_{\rho,\infty}(\mathbb H^n)$ into $WL^{q,(\kappa q)}_{\rho,\infty}(\mathbb H^n)$.
\end{thm}

Before stating our next theorem, we need to introduce a new space $\mathrm{BMO}_{\rho,\infty}(\mathbb H^n)$ defined by
\begin{equation*}
\mathrm{BMO}_{\rho,\infty}(\mathbb H^n):=\bigcup_{\theta>0}\mathrm{BMO}_{\rho,\theta}(\mathbb H^n),
\end{equation*}
where for $0<\theta<\infty$ the space $\mathrm{BMO}_{\rho,\theta}(\mathbb H^n)$ is defined to be the set of all locally integrable functions $f$ satisfying
\begin{equation}\label{BM}
\frac{1}{|B(u_0,r)|}\int_{B(u_0,r)}\big|f(u)-f_{B(u_0,r)}\big|\,du
\leq C\cdot\left(1+\frac{r}{\rho(u_0)}\right)^{\theta},
\end{equation}
for all $u_0\in\mathbb H^n$ and $r>0$, $f_{B(u_0,r)}$ denotes the mean value of $f$ on $B(u_0,r)$, that is,
\begin{equation*}
f_{B(u_0,r)}:=\frac{1}{|B(u_0,r)|}\int_{B(u_0,r)}f(v)\,dv.
\end{equation*}
A norm for $f\in\mathrm{BMO}_{\rho,\theta}(\mathbb H^n)$, denoted by $\|f\|_{\mathrm{BMO}_{\rho,\theta}}$, is given by the infimum of the constants satisfying \eqref{BM}, or equivalently,
\begin{equation*}
\|f\|_{\mathrm{BMO}_{\rho,\theta}}
:=\sup_{B(u_0,r)}\left(1+\frac{r}{\rho(u_0)}\right)^{-\theta}\bigg(\frac{1}{|B(u_0,r)|}\int_{B(u_0,r)}\big|f(u)-f_{B(u_0,r)}\big|\,du\bigg),
\end{equation*}
where the supremum is taken over all balls $B(u_0,r)$ with $u_0\in\mathbb H^n$ and $r>0$. Recall that in the setting of $\mathbb R^n$, the space $\mathrm{BMO}_{\rho,\theta}(\mathbb R^n)$ was first introduced by Bongioanni et al.\cite{bong2} (see also \cite{bong3}).

Moreover, given any $\beta\in[0,1]$, we introduce the space of H\"older continuous functions on $\mathbb H^n$, with exponent $\beta$.
\begin{equation*}
\mathcal{C}^{\beta}_{\rho,\infty}(\mathbb H^n):=\bigcup_{\theta>0}\mathcal{C}^{\beta}_{\rho,\theta}(\mathbb H^n),
\end{equation*}
where for $0<\theta<\infty$ the space $\mathcal{C}^{\beta}_{\rho,\theta}(\mathbb H^n)$ is defined to be the set of all locally integrable functions $f$ satisfying
\begin{equation}\label{hconti}
\frac{1}{|B(u_0,r)|^{1+\beta/Q}}\int_{B(u_0,r)}\big|f(u)-f_{B(u_0,r)}\big|\,du
\leq C\cdot\left(1+\frac{r}{\rho(u_0)}\right)^{\theta},
\end{equation}
for all $u_0\in\mathbb H^n$ and $r\in(0,\infty)$. The smallest bound $C$ for which \eqref{hconti} is satisfied is then taken to be the norm of $f$ in this space and is denoted by $\|f\|_{\mathcal{C}^{\beta}_{\rho,\theta}}$. When $\theta=0$ or $V\equiv0$, $\mathrm{BMO}_{\rho,\theta}(\mathbb H^n)$ and $\mathcal{C}^{\beta}_{\rho,\theta}(\mathbb H^n)$ will be simply written as $\mathrm{BMO}(\mathbb H^n)$ and $\mathcal{C}^{\beta}(\mathbb H^n)$, respectively. Note that when $\beta=0$ this space $\mathcal{C}^{\beta}_{\rho,\theta}(\mathbb H^n)$ reduces to the space $\mathrm{BMO}_{\rho,\theta}(\mathbb H^n)$ mentioned above.

For the case $\kappa\geq p/q$ of Theorem \ref{mainthm:1}, we will prove the following result.
\begin{thm}\label{mainthm:3}
Let $0<\alpha<Q$, $1<p<Q/{\alpha}$ and $1/q=1/p-{\alpha}/Q$. If $V\in RH_s$ with $s\geq Q/2$ and $p/q\leq\kappa<1$, then the $\mathcal L$-fractional integral operator $\mathcal I_{\alpha}$ is bounded from $L^{p,\kappa}_{\rho,\infty}(\mathbb H^n)$ into $\mathcal{C}^{\beta}_{\rho,\infty}(\mathbb H^n)$ with $\beta/Q=\kappa/p-1/q$ and $\beta$ sufficiently small. To be more precise, $\beta<\delta\leq1$ and $\delta$ is given as in Lemma $\ref{kernel2}$.
\end{thm}

In particular, for the limiting case $\kappa=p/q$ (or $\beta=0$), we obtain the following result on BMO-type estimate of $\mathcal I_{\alpha}$.
\begin{cor}\label{mainthm:4}
Let $0<\alpha<Q$, $1<p<Q/{\alpha}$ and $1/q=1/p-{\alpha}/Q$. If $V\in RH_s$ with $s\geq Q/2$ and $\kappa=p/q$, then the $\mathcal L$-fractional integral operator $\mathcal I_{\alpha}$ is bounded from $L^{p,\kappa}_{\rho,\infty}(\mathbb H^n)$ into $\mathrm{BMO}_{\rho,\infty}(\mathbb H^n)$.
\end{cor}

\section{Proofs of Theorems $\ref{mainthm:1}$ and $\ref{mainthm:2}$}
In this section, we will prove the conclusions of Theorems \ref{mainthm:1} and \ref{mainthm:2}. Let us remind that the $\mathcal L$-fractional integral operator of order $\alpha\in(0,Q)$ can be written as
\begin{equation*}
\mathcal I_{\alpha}f(u)={\mathcal L}^{-{\alpha}/2}f(u)=\int_{\mathbb H^n}\mathcal K_{\alpha}(u,v)f(v)\,dv,
\end{equation*}
where
\begin{equation}\label{kauv}
\mathcal K_{\alpha}(u,v)=\frac{1}{\Gamma(\alpha/2)}\int_0^{\infty}P_s(u,v)\,s^{\alpha/2-1}ds.
\end{equation}

The following lemma gives the estimate of the kernel $\mathcal K_{\alpha}(u,v)$ related to the Schr\"odinger operator $\mathcal L$, which plays a key role in the proof of our main theorems.
\begin{lem}\label{kernel}
Let $V\in RH_s$ with $s\geq Q/2$ and $0<\alpha<Q$. For every positive integer $N\geq1$, there exists a positive constant $C_{N,\alpha}>0$ such that for all $u$ and $v$ in $\mathbb H^n$,
\begin{equation}\label{WH1}
\big|\mathcal K_{\alpha}(u,v)\big|\leq C_{N,\alpha}\bigg(1+\frac{|v^{-1}u|}{\rho(u)}\bigg)^{-N}\frac{1}{|v^{-1}u|^{Q-\alpha}}.
\end{equation}
\end{lem}
\begin{proof}
By Lemma \ref{ker1} and \eqref{kauv}, we have
\begin{equation*}
\begin{split}
\big|\mathcal K_{\alpha}(u,v)\big|&\leq\frac{1}{\Gamma(\alpha/2)}\int_0^{\infty}\big|P_s(u,v)\big|\,s^{\alpha/2-1}ds\\
&\leq\frac{1}{\Gamma(\alpha/2)}\int_0^{\infty}\frac{C_N}{s^{Q/2}}\cdot\exp\bigg(-\frac{|v^{-1}u|^2}{As}\bigg)
\bigg(1+\frac{\sqrt{s\,}}{\rho(u)}+\frac{\sqrt{s\,}}{\rho(v)}\bigg)^{-N}s^{\alpha/2-1}ds\\
&\leq\frac{1}{\Gamma(\alpha/2)}\int_0^{\infty}\frac{C_N}{s^{Q/2}}\cdot\exp\bigg(-\frac{|v^{-1}u|^2}{As}\bigg)
\bigg(1+\frac{\sqrt{s\,}}{\rho(u)}\bigg)^{-N}s^{\alpha/2-1}ds.
\end{split}
\end{equation*}
We consider two cases $s>|v^{-1}u|^2$ and $0\leq s\leq|v^{-1}u|^2$, respectively. Thus, $|\mathcal K_{\alpha}(u,v)|\leq I+II$, where
\begin{equation*}
I=\frac{1}{\Gamma(\alpha/2)}\int_{|v^{-1}u|^2}^{\infty}\frac{C_N}{s^{Q/2}}\cdot\exp\bigg(-\frac{|v^{-1}u|^2}{As}\bigg)
\bigg(1+\frac{\sqrt{s\,}}{\rho(u)}\bigg)^{-N}s^{\alpha/2-1}ds
\end{equation*}
and
\begin{equation*}
II=\frac{1}{\Gamma(\alpha/2)}\int_0^{|v^{-1}u|^2}\frac{C_N}{s^{Q/2}}\cdot\exp\bigg(-\frac{|v^{-1}u|^2}{As}\bigg)
\bigg(1+\frac{\sqrt{s\,}}{\rho(u)}\bigg)^{-N}s^{\alpha/2-1}ds.
\end{equation*}
When $s>|v^{-1}u|^2$, then $\sqrt{s\,}>|v^{-1}u|$, and hence
\begin{equation*}
\begin{split}
I&\leq\frac{1}{\Gamma(\alpha/2)}\int_{|v^{-1}u|^2}^{\infty}\frac{C_N}{s^{Q/2}}\cdot\exp\bigg(-\frac{|v^{-1}u|^2}{As}\bigg)
\bigg(1+\frac{|v^{-1}u|}{\rho(u)}\bigg)^{-N}s^{\alpha/2-1}ds\\
&\leq C_{N,\alpha}\bigg(1+\frac{|v^{-1}u|}{\rho(u)}\bigg)^{-N}\int_{|v^{-1}u|^2}^{\infty}s^{\alpha/2-Q/2-1}ds\\
&\leq C_{N,\alpha}\bigg(1+\frac{|v^{-1}u|}{\rho(u)}\bigg)^{-N}\frac{1}{|v^{-1}u|^{Q-\alpha}},
\end{split}
\end{equation*}
where the last integral converges because $0<\alpha<Q$. On the other hand,
\begin{equation*}
\begin{split}
II&\leq C_{N,\alpha}\int_0^{|v^{-1}u|^2}\frac{1}{s^{Q/2}}\cdot\bigg(\frac{|v^{-1}u|^2}{s}\bigg)^{-(Q/2+N/2)}
\bigg(1+\frac{\sqrt{s\,}}{\rho(u)}\bigg)^{-N}s^{\alpha/2-1}ds\\
&=C_{N,\alpha}\int_0^{|v^{-1}u|^2}\frac{1}{|v^{-1}u|^Q}\cdot\bigg(\frac{\sqrt{s\,}}{|v^{-1}u|}\bigg)^{N}
\bigg(1+\frac{\sqrt{s\,}}{\rho(u)}\bigg)^{-N}s^{\alpha/2-1}ds.
\end{split}
\end{equation*}
It is easy to see that when $0\leq s\leq|v^{-1}u|^2$,
\begin{equation*}
\frac{\sqrt{s\,}}{|v^{-1}u|}\leq\frac{\sqrt{s\,}+\rho(u)}{|v^{-1}u|+\rho(u)}.
\end{equation*}
Hence,
\begin{equation*}
\begin{split}
II&\leq C_{N,\alpha}\int_0^{|v^{-1}u|^2}\frac{1}{|v^{-1}u|^Q}\cdot\bigg(\frac{\sqrt{s\,}+\rho(u)}{|v^{-1}u|+\rho(u)}\bigg)^{N}
\bigg(\frac{\sqrt{s\,}+\rho(u)}{\rho(u)}\bigg)^{-N}s^{\alpha/2-1}ds\\
&=\frac{C_{N,\alpha}}{|v^{-1}u|^Q}\bigg(1+\frac{|v^{-1}u|}{\rho(u)}\bigg)^{-N}\int_0^{|v^{-1}u|^2}s^{\alpha/2-1}ds\\
&=C_{N,\alpha}\bigg(1+\frac{|v^{-1}u|}{\rho(u)}\bigg)^{-N}\frac{1}{|v^{-1}u|^{Q-\alpha}}.
\end{split}
\end{equation*}
Combining the estimates of $I$ and $II$ yields the desired estimate \eqref{WH1} for $\alpha\in(0,Q)$. This concludes the proof of the lemma.
\end{proof}
We are now ready to show our main theorems.
\begin{proof}[Proof of Theorem $\ref{mainthm:1}$]
By definition, we only need to show that for any given ball $B=B(u_0,r)$ of $\mathbb H^n$, there is some $\vartheta>0$ such that
\begin{equation}\label{Main1}
\bigg(\frac{1}{|B|^{\kappa q/p}}\int_B\big|\mathcal I_{\alpha}f(u)\big|^q\,du\bigg)^{1/q}\leq C\cdot\left(1+\frac{r}{\rho(u_0)}\right)^{\vartheta}
\end{equation}
holds for given $f\in L^{p,\kappa}_{\rho,\infty}(\mathbb H^n)$ with $(p,\kappa)\in(1,Q/{\alpha})\times(0,p/q)$. Suppose that $f\in L^{p,\kappa}_{\rho,\theta}(\mathbb H^n)$ for some $\theta>0$. We decompose the function $f$ as
\begin{equation*}
\begin{cases}
f=f_1+f_2\in L^{p,\kappa}_{\rho,\theta}(\mathbb H^n);\  &\\
f_1=f\cdot\chi_{2B};\  &\\
f_2=f\cdot\chi_{(2B)^c},
\end{cases}
\end{equation*}
where $2B$ is the ball centered at $u_0$ of radius $2r>0$, $\chi_{2B}$ is the characteristic function of $2B$ and $(2B)^c=\mathbb H^n\backslash(2B)$. Then by the linearity of $\mathcal I_{\alpha}$, we write
\begin{equation*}
\begin{split}
\bigg(\frac{1}{|B|^{\kappa q/p}}\int_B\big|\mathcal I_{\alpha}f(u)\big|^q\,du\bigg)^{1/q}
&\leq\bigg(\frac{1}{|B|^{\kappa q/p}}\int_B\big|\mathcal I_{\alpha}f_1(u)\big|^q\,du\bigg)^{1/q}\\
&+\bigg(\frac{1}{|B|^{\kappa q/p}}\int_B\big|\mathcal I_{\alpha}f_2(u)\big|^q\,du\bigg)^{1/q}\\
&:=I_1+I_2.
\end{split}
\end{equation*}
In what follows, we consider each part separately. By Theorem \ref{strong} (1), we have
\begin{equation*}
\begin{split}
I_1&=\bigg(\frac{1}{|B|^{\kappa q/p}}\int_B\big|\mathcal I_{\alpha}f_1(u)\big|^q\,du\bigg)^{1/q}\\
&\leq C\cdot\frac{1}{|B|^{\kappa/p}}\bigg(\int_{\mathbb H^n}\big|f_1(u)\big|^p\,du\bigg)^{1/p}\\
&=C\cdot\frac{1}{|B|^{\kappa/p}}\bigg(\int_{2B}\big|f(u)\big|^p\,du\bigg)^{1/p}\\
&\leq C\big\|f\big\|_{L^{p,\kappa}_{\rho,\theta}(\mathbb H^n)}\cdot
\frac{|2B|^{\kappa/p}}{|B|^{\kappa/p}}\cdot\left(1+\frac{2r}{\rho(u_0)}\right)^{\theta}.
\end{split}
\end{equation*}
Also observe that for any fixed $\theta>0$,
\begin{equation}\label{2rx}
1\leq\left(1+\frac{2r}{\rho(u_0)}\right)^{\theta}\leq 2^{\theta}\left(1+\frac{r}{\rho(u_0)}\right)^{\theta}.
\end{equation}
This in turn implies that
\begin{equation*}
\begin{split}
I_1&\leq C_{\theta,n}\big\|f\big\|_{L^{p,\kappa}_{\rho,\theta}(\mathbb H^n)}\left(1+\frac{r}{\rho(u_0)}\right)^{\theta}.
\end{split}
\end{equation*}
Next we estimate the other term $I_2$. Notice that for any $u\in B(u_0,r)$ and $v\in (2B)^c$, one has
\begin{equation*}
\big|v^{-1}u\big|=\big|(v^{-1}u_0)\cdot(u_0^{-1}u)\big|\leq\big|v^{-1}u_0\big|+\big|u_0^{-1}u\big|
\end{equation*}
and
\begin{equation*}
\big|v^{-1}u\big|=\big|(v^{-1}u_0)\cdot(u_0^{-1}u)\big|\geq\big|v^{-1}u_0\big|-\big|u_0^{-1}u\big|.
\end{equation*}
Thus,
\begin{equation*}
\frac{1}{\,2\,}\big|v^{-1}u_0\big|\leq\big|v^{-1}u\big|\leq\frac{3}{\,2\,}\big|v^{-1}u_0\big|,
\end{equation*}
i.e., $|v^{-1}u|\approx|v^{-1}u_0|$. It then follows from Lemma \ref{kernel} that for any $u\in B(u_0,r)$ and any positive integer $N$,
\begin{equation}\label{Talpha}
\begin{split}
\big|\mathcal I_{\alpha}f_2(u)\big|&\leq\int_{(2B)^c}|\mathcal K_{\alpha}(u,v)|\cdot|f(v)|\,dv\\
&\leq C_{N,\alpha}\int_{(2B)^c}\bigg(1+\frac{|v^{-1}u|}{\rho(u)}\bigg)^{-N}\frac{1}{|v^{-1}u|^{Q-\alpha}}\cdot|f(v)|\,dv\\
&\leq C_{N,\alpha,n}\int_{(2B)^c}\bigg(1+\frac{|v^{-1}u_0|}{\rho(u)}\bigg)^{-N}\frac{1}{|v^{-1}u_0|^{Q-\alpha}}\cdot|f(v)|\,dv\\
&=C_{N,\alpha,n}\sum_{k=1}^\infty\int_{2^kr\leq|v^{-1}u_0|<2^{k+1}r}\bigg(1+\frac{|v^{-1}u_0|}{\rho(u)}\bigg)^{-N}
\frac{1}{|v^{-1}u_0|^{Q-\alpha}}\cdot|f(v)|\,dv\\
&\leq C_{N,\alpha,n}\sum_{k=1}^\infty\frac{1}{|B(u_0,2^{k+1}r)|^{1-(\alpha/Q)}}
\int_{|v^{-1}u_0|<2^{k+1}r}\bigg(1+\frac{2^kr}{\rho(u)}\bigg)^{-N}|f(v)|\,dv.
\end{split}
\end{equation}
In view of \eqref{com2} and \eqref{2rx}, we can further obtain
\begin{align}\label{Tf2}
\big|\mathcal I_{\alpha}f_2(u)\big|
&\leq C\sum_{k=1}^\infty\frac{1}{|B(u_0,2^{k+1}r)|^{1-(\alpha/Q)}}\notag\\
&\times\int_{|v^{-1}u_0|<2^{k+1}r}\left(1+\frac{r}{\rho(u_0)}\right)^{N\cdot\frac{N_0}{N_0+1}}
\left(1+\frac{2^kr}{\rho(u_0)}\right)^{-N}|f(v)|\,dv\notag\\
&\leq C\sum_{k=1}^\infty\frac{1}{|B(u_0,2^{k+1}r)|^{1-(\alpha/Q)}}\notag\\
&\times\int_{B(u_0,2^{k+1}r)}\left(1+\frac{r}{\rho(u_0)}\right)^{N\cdot\frac{N_0}{N_0+1}}
\left(1+\frac{2^{k+1}r}{\rho(u_0)}\right)^{-N}|f(v)|\,dv.
\end{align}
We consider each term in the sum of \eqref{Tf2} separately. By using H\"older's inequality, we obtain that for each integer $k\geq1$,
\begin{equation*}
\begin{split}
&\frac{1}{|B(u_0,2^{k+1}r)|^{1-(\alpha/Q)}}\int_{B(u_0,2^{k+1}r)}\big|f(v)\big|\,dv\\
&\leq\frac{1}{|B(u_0,2^{k+1}r)|^{1-(\alpha/Q)}}\bigg(\int_{B(u_0,2^{k+1}r)}\big|f(v)\big|^p\,dv\bigg)^{1/p}
\bigg(\int_{B(u_0,2^{k+1}r)}1\,dv\bigg)^{1/{p'}}\\
&\leq C\big\|f\big\|_{L^{p,\kappa}_{\rho,\theta}(\mathbb H^n)}\cdot\frac{|B(u_0,2^{k+1}r)|^{{\kappa}/p}}{|B(u_0,2^{k+1}r)|^{1/q}}
\left(1+\frac{2^{k+1}r}{\rho(u_0)}\right)^{\theta}.
\end{split}
\end{equation*}
This allows us to obtain
\begin{equation*}
\begin{split}
I_2&\leq C\big\|f\big\|_{L^{p,\kappa}_{\rho,\theta}(\mathbb H^n)}\cdot\frac{|B(u_0,r)|^{1/q}}{|B(u_0,r)|^{{\kappa}/p}}
\sum_{k=1}^\infty\frac{|B(u_0,2^{k+1}r)|^{{\kappa}/p}}{|B(u_0,2^{k+1}r)|^{1/q}}
\left(1+\frac{r}{\rho(u_0)}\right)^{N\cdot\frac{N_0}{N_0+1}}\left(1+\frac{2^{k+1}r}{\rho(u_0)}\right)^{-N+\theta}\\
&=C\big\|f\big\|_{L^{p,\kappa}_{\rho,\theta}(\mathbb H^n)}\left(1+\frac{r}{\rho(u_0)}\right)^{N\cdot\frac{N_0}{N_0+1}}
\sum_{k=1}^\infty\frac{|B(u_0,r)|^{1/q-\kappa/p}}{|B(u_0,2^{k+1}r)|^{1/q-\kappa/p}}
\left(1+\frac{2^{k+1}r}{\rho(u_0)}\right)^{-N+\theta}.
\end{split}
\end{equation*}
Thus, by choosing $N$ large enough so that $N>\theta$, and the last series is convergent, then we have
\begin{equation*}
\begin{split}
I_2&\leq C\big\|f\big\|_{L^{p,\kappa}_{\rho,\theta}(\mathbb H^n)}
\left(1+\frac{r}{\rho(u_0)}\right)^{N\cdot\frac{N_0}{N_0+1}}\sum_{k=1}^\infty\left(\frac{|B(u_0,r)|}{|B(u_0,2^{k+1}r)|}\right)^{{(1/q-\kappa/p)}}\\
&\leq C\big\|f\big\|_{L^{p,\kappa}_{\rho,\theta}(\mathbb H^n)}\left(1+\frac{r}{\rho(u_0)}\right)^{N\cdot\frac{N_0}{N_0+1}},
\end{split}
\end{equation*}
where the last inequality follows from the fact that $1/q-\kappa/p>0$. Summing up the above estimates for $I_1$ and $I_2$ and letting $\vartheta=\max\big\{\theta,N\cdot\frac{N_0}{N_0+1}\big\}$, we obtain the desired inequality \eqref{Main1}. This completes the proof of Theorem \ref{mainthm:1}.
\end{proof}

\begin{proof}[Proof of Theorem $\ref{mainthm:2}$]
To prove Theorem \ref{mainthm:2}, by definition, it suffices to prove that for each given ball $B=B(u_0,r)$ of $\mathbb H^n$, there is some $\vartheta>0$ such that
\begin{equation}\label{Main2}
\frac{1}{|B|^{\kappa}}\sup_{\lambda>0}\lambda\cdot\big|\big\{u\in B:|\mathcal I_{\alpha}f(u)|>\lambda\big\}\big|^{1/q}
\leq C\cdot\left(1+\frac{r}{\rho(u_0)}\right)^{\vartheta}
\end{equation}
holds for given $f\in L^{1,\kappa}_{\rho,\infty}(\mathbb H^n)$ with $0<\kappa<1/q$ and $q=Q/{(Q-\alpha)}$. Now suppose that $f\in L^{1,\kappa}_{\rho,\theta}(\mathbb H^n)$ for some $\theta>0$. We decompose the function $f$ as
\begin{equation*}
\begin{cases}
f=f_1+f_2\in L^{1,\kappa}_{\rho,\theta}(\mathbb H^n);\  &\\
f_1=f\cdot\chi_{2B};\  &\\
f_2=f\cdot\chi_{(2B)^c}.
\end{cases}
\end{equation*}
Then for any given $\lambda>0$, by the linearity of $\mathcal I_{\alpha}$, we can write
\begin{equation*}
\begin{split}
&\frac{1}{|B|^{\kappa}}\lambda\cdot\big|\big\{u\in B:|\mathcal I_{\alpha}f(u)|>\lambda\big\}\big|^{1/q}\\
&\leq\frac{1}{|B|^{\kappa}}\lambda\cdot\big|\big\{u\in B:|\mathcal I_{\alpha}f_1(u)|>\lambda/2\big\}\big|^{1/q}\\
&+\frac{1}{|B|^{\kappa}}\lambda\cdot\big|\big\{u\in B:|\mathcal I_{\alpha}f_2(u)|>\lambda/2\big\}\big|^{1/q}\\
&:=J_1+J_2.
\end{split}
\end{equation*}
We first give the estimate for the term $J_1$. By Theorem \ref{strong} (2), we get
\begin{equation*}
\begin{split}
J_1&=\frac{1}{|B|^{\kappa}}\lambda\cdot\big|\big\{u\in B:|\mathcal I_{\alpha} f_1(u)|>\lambda/2\big\}\big|^{1/q}\\
&\leq C\cdot\frac{1}{|B|^{\kappa}}\bigg(\int_{\mathbb H^n}\big|f_1(u)\big|\,du\bigg)\\
&=C\cdot\frac{1}{|B|^{\kappa}}\bigg(\int_{2B}\big|f(u)\big|\,du\bigg)\\
&\leq C\big\|f\big\|_{L^{1,\kappa}_{\rho,\theta}(\mathbb H^n)}\cdot\frac{|2B|^{\kappa}}{|B|^{\kappa}}\left(1+\frac{2r}{\rho(u_0)}\right)^{\theta}.
\end{split}
\end{equation*}
Therefore, in view of \eqref{2rx},
\begin{equation*}
J_1\leq C\big\|f\big\|_{L^{1,\kappa}_{\rho,\theta}(\mathbb H^n)}\cdot\left(1+\frac{r}{\rho(u_0)}\right)^{\theta}.
\end{equation*}
As for the second term $J_2$, by using the pointwise inequality \eqref{Tf2} and Chebyshev's inequality, we can deduce that
\begin{equation}\label{Tf2pr}
\begin{split}
J_2&=\frac{1}{|B|^{\kappa}}\lambda\cdot\big|\big\{u\in B:|\mathcal I_{\alpha}f_2(u)|>\lambda/2\big\}\big|^{1/q}\\
&\leq\frac{2}{|B|^{\kappa}}\bigg(\int_{B}\big|\mathcal I_{\alpha}f_2(u)\big|^q\,du\bigg)^{1/q}\\
&\leq C\cdot\frac{|B|^{1/q}}{|B|^{\kappa}}
\sum_{k=1}^\infty\frac{1}{|B(u_0,2^{k+1}r)|^{1-(\alpha/Q)}}\\
&\times\int_{B(u_0,2^{k+1}r)}\left(1+\frac{r}{\rho(u_0)}\right)^{N\cdot\frac{N_0}{N_0+1}}
\left(1+\frac{2^{k+1}r}{\rho(u_0)}\right)^{-N}|f(v)|\,dv.
\end{split}
\end{equation}
We consider each term in the sum of \eqref{Tf2pr} separately. For each integer $k\geq1$, we compute
\begin{equation*}
\begin{split}
&\frac{1}{|B(u_0,2^{k+1}r)|^{1-(\alpha/Q)}}\int_{B(u_0,2^{k+1}r)}\big|f(v)\big|\,dv\\
&\leq C\big\|f\big\|_{L^{1,\kappa}_{\rho,\theta}(\mathbb H^n)}\cdot
\frac{|B(u_0,2^{k+1}r)|^{\kappa}}{|B(u_0,2^{k+1}r)|^{1/q}}\left(1+\frac{2^{k+1}r}{\rho(u_0)}\right)^{\theta}.
\end{split}
\end{equation*}
Consequently,
\begin{equation*}
\begin{split}
J_2&\leq C\big\|f\big\|_{L^{1,\kappa}_{\rho,\theta}(\mathbb H^n)}
\cdot\frac{|B(u_0,r)|^{1/q}}{|B(u_0,r)|^{\kappa}}\sum_{k=1}^\infty\frac{|B(u_0,2^{k+1}r)|^{\kappa}}{|B(u_0,2^{k+1}r)|^{1/q}}
\left(1+\frac{r}{\rho(u_0)}\right)^{N\cdot\frac{N_0}{N_0+1}}\left(1+\frac{2^{k+1}r}{\rho(u_0)}\right)^{-N+\theta}\\
&=C\big\|f\big\|_{L^{1,\kappa}_{\rho,\theta}(\mathbb H^n)}
\left(1+\frac{r}{\rho(u_0)}\right)^{N\cdot\frac{N_0}{N_0+1}}\sum_{k=1}^\infty\frac{|B(u_0,r)|^{{1/q-\kappa}}}{|B(u_0,2^{k+1}r)|^{{1/q-\kappa}}}
\left(1+\frac{2^{k+1}r}{\rho(u_0)}\right)^{-N+\theta}.
\end{split}
\end{equation*}
Therefore, by selecting $N$ large enough so that $N>\theta$, we thus have
\begin{equation*}
\begin{split}
J_2&\leq C\big\|f\big\|_{L^{1,\kappa}_{\rho,\theta}(\mathbb H^n)}\left(1+\frac{r}{\rho(u_0)}\right)^{N\cdot\frac{N_0}{N_0+1}}
\sum_{k=1}^\infty\left(\frac{|B(u_0,r)|}{|B(u_0,2^{k+1}r)|}\right)^{{(1/q-\kappa)}}\\
&\leq C\big\|f\big\|_{L^{1,\kappa}_{\rho,\theta}(\mathbb H^n)}
\left(1+\frac{r}{\rho(u_0)}\right)^{N\cdot\frac{N_0}{N_0+1}},
\end{split}
\end{equation*}
where the last step is due to the fact that $0<\kappa<1/q$. Let $\vartheta=\max\big\{\theta,N\cdot\frac{N_0}{N_0+1}\big\}$. Here $N$ is an appropriate constant. Summing up the above estimates for $J_1$ and $J_2$, and then taking the supremum over all $\lambda>0$, we obtain the desired inequality \eqref{Main2}. This finishes the proof of Theorem \ref{mainthm:2}.
\end{proof}

\section{Proof of Theorem \ref{mainthm:3}}
We need the following lemma which establishes the Lipschitz regularity of the kernel $P_s(u,v)$. See Lemma 11 and Remark 4 in \cite{lin}.
\begin{lem}[\cite{lin}]\label{ker2}
Let $V\in RH_s$ with $s\geq Q/2$. For every positive integer $N\geq1$, there exists a positive constant $C_N>0$ such that for all $u$ and $v$ in $\mathbb H^n$, and for some fixed $0<\delta\leq 1$,
\begin{equation*}
\big|P_s(u\cdot h,v)-P_s(u,v)\big|\leq C_N\bigg(\frac{|h|}{\sqrt{s\,}}\bigg)^{\delta} s^{-Q/2}\exp\bigg(-\frac{|v^{-1}u|^2}{As}\bigg)\bigg(1+\frac{\sqrt{s\,}}{\rho(u)}+\frac{\sqrt{s\,}}{\rho(v)}\bigg)^{-N},
\end{equation*}
whenever $|h|\leq|v^{-1}u|/2$.
\end{lem}

Based on the above lemma, we are able to prove the following result, which plays a key role in the proof of our main theorem.

\begin{lem}\label{kernel2}
Let $V\in RH_s$ with $s\geq Q/2$ and $0<\alpha<Q$. For every positive integer $N\geq1$, there exists a positive constant $C_{N,\alpha}>0$ such that for all $u,v$ and $w$ in $\mathbb H^n$, and for some fixed $0<\delta\leq 1$,
\begin{equation}\label{WH2}
\big|\mathcal K_{\alpha}(u,w)-\mathcal K_{\alpha}(v,w)\big|\leq C_{N,\alpha}\bigg(1+\frac{|w^{-1}u|}{\rho(u)}\bigg)^{-N}\frac{|v^{-1}u|^{\delta}}{|w^{-1}u|^{Q-\alpha+\delta}},
\end{equation}
whenever $|v^{-1}u|\leq |w^{-1}u|/2$.
\end{lem}
\begin{proof}
In view of Lemma \ref{ker2} and \eqref{kauv}, we have
\begin{equation*}
\begin{split}
&\big|\mathcal K_{\alpha}(u,w)-\mathcal K_{\alpha}(v,w)\big|\\
&=\frac{1}{\Gamma(\alpha/2)}\bigg|\int_0^{\infty}P_s(u,w)\,s^{\alpha/2-1}ds-\int_0^{\infty}P_s(v,w)\,s^{\alpha/2-1}ds\bigg|\\
&\leq\frac{1}{\Gamma(\alpha/2)}\int_0^{\infty}\big|P_s(u\cdot(u^{-1}v),w)-P_s(u,w)\big|\,s^{\alpha/2-1}ds\\
&\leq\frac{1}{\Gamma(\alpha/2)}\int_0^{\infty}C_N\cdot\bigg(\frac{|u^{-1}v|}{\sqrt{s\,}}\bigg)^{\delta} s^{-Q/2}\exp\bigg(-\frac{|w^{-1}u|^2}{As}\bigg)\bigg(1+\frac{\sqrt{s\,}}{\rho(u)}+\frac{\sqrt{s\,}}{\rho(w)}\bigg)^{-N}s^{\alpha/2-1}ds\\
&\leq\frac{1}{\Gamma(\alpha/2)}\int_0^{\infty}C_N\cdot\bigg(\frac{|u^{-1}v|}{\sqrt{s\,}}\bigg)^{\delta} s^{-Q/2}\exp\bigg(-\frac{|w^{-1}u|^2}{As}\bigg)\bigg(1+\frac{\sqrt{s\,}}{\rho(u)}\bigg)^{-N}s^{\alpha/2-1}ds.
\end{split}
\end{equation*}
Arguing as in the proof of Lemma \ref{kernel}, consider two cases as below: $s>|w^{-1}u|^2$ and $0\leq s\leq|w^{-1}u|^2$. Then the right-hand side of the above expression can be written as $III+IV$, where
\begin{equation*}
III=\frac{1}{\Gamma(\alpha/2)}\int_{|w^{-1}u|^2}^{\infty}\frac{C_N}{s^{Q/2}}\cdot
\bigg(\frac{|u^{-1}v|}{\sqrt{s\,}}\bigg)^{\delta}\exp\bigg(-\frac{|w^{-1}u|^2}{As}\bigg)
\bigg(1+\frac{\sqrt{s\,}}{\rho(u)}\bigg)^{-N}s^{\alpha/2-1}ds,
\end{equation*}
and
\begin{equation*}
IV=\frac{1}{\Gamma(\alpha/2)}\int_0^{|w^{-1}u|^2}\frac{C_N}{s^{Q/2}}\cdot
\bigg(\frac{|u^{-1}v|}{\sqrt{s\,}}\bigg)^{\delta}\exp\bigg(-\frac{|w^{-1}u|^2}{As}\bigg)
\bigg(1+\frac{\sqrt{s\,}}{\rho(u)}\bigg)^{-N}s^{\alpha/2-1}ds.
\end{equation*}
When $s>|w^{-1}u|^2$, then $\sqrt{s\,}>|w^{-1}u|$, and hence
\begin{equation*}
\begin{split}
III&\leq\frac{1}{\Gamma(\alpha/2)}\int_{|w^{-1}u|^2}^{\infty}\frac{C_N}{s^{Q/2}}\cdot\bigg(\frac{|u^{-1}v|}{|w^{-1}u|}\bigg)^{\delta}
\exp\bigg(-\frac{|w^{-1}u|^2}{As}\bigg)\bigg(1+\frac{|w^{-1}u|}{\rho(u)}\bigg)^{-N}s^{\alpha/2-1}ds\\
&\leq C_{N,\alpha}\bigg(1+\frac{|w^{-1}u|}{\rho(u)}\bigg)^{-N}\bigg(\frac{|u^{-1}v|}{|w^{-1}u|}\bigg)^{\delta}
\int_{|w^{-1}u|^2}^{\infty}s^{\alpha/2-Q/2-1}ds\\
&\leq C_{N,\alpha}\bigg(1+\frac{|w^{-1}u|}{\rho(u)}\bigg)^{-N}\frac{|v^{-1}u|^{\delta}}{|w^{-1}u|^{Q-\alpha+\delta}},
\end{split}
\end{equation*}
where the last inequality holds since $|u^{-1}v|=|v^{-1}u|$ and $0<\alpha<Q$. On the other hand,
\begin{equation*}
\begin{split}
IV&\leq C_{N,\alpha}\int_0^{|w^{-1}u|^2}\frac{1}{s^{Q/2}}\cdot\bigg(\frac{|u^{-1}v|}{\sqrt{s\,}}\bigg)^{\delta}
\bigg(\frac{|w^{-1}u|^2}{s}\bigg)^{-(Q/2+N/2+\delta/2)}\bigg(1+\frac{\sqrt{s\,}}{\rho(u)}\bigg)^{-N}s^{\alpha/2-1}ds\\
&=C_{N,\alpha}\int_0^{|w^{-1}u|^2}\frac{|u^{-1}v|^{\delta}}{|w^{-1}u|^{Q+\delta}}\bigg(\frac{\sqrt{s\,}}{|w^{-1}u|}\bigg)^{N}
\bigg(1+\frac{\sqrt{s\,}}{\rho(u)}\bigg)^{-N}s^{\alpha/2-1}ds.
\end{split}
\end{equation*}
It is easy to check that when $0\leq s\leq|w^{-1}u|^2$,
\begin{equation*}
\frac{\sqrt{s\,}}{|w^{-1}u|}\leq\frac{\sqrt{s\,}+\rho(u)}{|w^{-1}u|+\rho(u)}.
\end{equation*}
This in turn implies that
\begin{equation*}
\begin{split}
IV&\leq C_{N,\alpha}\int_0^{|w^{-1}u|^2}\frac{|u^{-1}v|^{\delta}}{|w^{-1}u|^{Q+\delta}}\bigg(\frac{\sqrt{s\,}+\rho(u)}{|w^{-1}u|+\rho(u)}\bigg)^{N}
\bigg(\frac{\sqrt{s\,}+\rho(u)}{\rho(u)}\bigg)^{-N}s^{\alpha/2-1}ds\\
&=C_{N,\alpha}\cdot\frac{|u^{-1}v|^{\delta}}{|w^{-1}u|^{Q+\delta}}\bigg(1+\frac{|w^{-1}u|}{\rho(u)}\bigg)^{-N}\int_0^{|w^{-1}u|^2}s^{\alpha/2-1}ds\\
&=C_{N,\alpha}\bigg(1+\frac{|w^{-1}u|}{\rho(u)}\bigg)^{-N}\frac{|v^{-1}u|^{\delta}}{|w^{-1}u|^{Q-\alpha+\delta}},
\end{split}
\end{equation*}
where the last step holds because $|u^{-1}v|=|v^{-1}u|$. Combining the estimates of $III$ and $IV$ produces the desired inequality \eqref{WH2} for $\alpha\in(0,Q)$. This concludes the proof of the lemma.
\end{proof}

We are now in a position to give the proof of Theorem $\ref{mainthm:3}$.
\begin{proof}[Proof of Theorem $\ref{mainthm:3}$]
Fix a ball $B=B(u_0,r)$ with $u_0\in\mathbb H^n$ and $r\in(0,\infty)$, it suffices to prove that the following inequality
\begin{equation}\label{end1.1}
\frac{1}{|B|^{1+\beta/Q}}\int_B\big|\mathcal I_{\alpha}f(u)-(\mathcal I_{\alpha}f)_B\big|\,du\leq C\cdot\left(1+\frac{r}{\rho(u_0)}\right)^{\vartheta}
\end{equation}
holds for given $f\in L^{p,\kappa}_{\rho,\infty}(\mathbb H^n)$ with $1<p<q<\infty$ and $p/q\leq\kappa<1$, where $0<\alpha<Q$ and $(\mathcal I_{\alpha}f)_B$ denotes the average of $\mathcal I_{\alpha}f$ over $B$. Suppose that $f\in L^{p,\kappa}_{\rho,\theta}(\mathbb H^n)$ for some $\theta>0$. Decompose the function $f$ as $f=f_1+f_2$, where $f_1=f\cdot\chi_{4B}$, $f_2=f\cdot\chi_{(4B)^c}$, $4B=B(u_0,4r)$ and $(4B)^c=\mathbb H^n\backslash(4B)$. By the linearity of the $\mathcal L$-fractional integral operator $\mathcal I_{\alpha}$, the left-hand side of \eqref{end1.1} can be written as
\begin{equation*}
\begin{split}
&\frac{1}{|B|^{1+\beta/Q}}\int_B\big|\mathcal I_{\alpha}f(u)-(\mathcal I_{\alpha}f)_B\big|\,du\\
&\leq\frac{1}{|B|^{1+\beta/Q}}\int_B\big|\mathcal I_{\alpha}f_1(u)-(\mathcal I_{\alpha}f_1)_B\big|\,du
+\frac{1}{|B|^{1+\beta/Q}}\int_B\big|\mathcal I_{\alpha}f_2(u)-(\mathcal I_{\alpha}f_2)_B\big|\,du\\
&:=K_1+K_2.
\end{split}
\end{equation*}
First let us consider the term $K_1$. Applying the strong-type $(p,q)$ estimate of $\mathcal I_{\alpha}$ (see Theorem \ref{strong}) and H\"older's inequality, we obtain
\begin{equation*}
\begin{split}
K_1&\leq\frac{2}{|B|^{1+\beta/Q}}\int_B|\mathcal I_{\alpha}f_1(u)|\,du\\
&\leq\frac{2}{|B|^{1+\beta/Q}}\bigg(\int_B|\mathcal I_{\alpha}f_1(u)|^q\,du\bigg)^{1/q}\bigg(\int_B1\,du\bigg)^{1/{q'}}\\
&\leq\frac{C}{|B|^{1+\beta/Q}}\bigg(\int_{4B}|f(u)|^p\,du\bigg)^{1/p}|B|^{1/{q'}}\\
&\leq C\big\|f\big\|_{L^{p,\kappa}_{\rho,\theta}(\mathbb H^n)}
\cdot\frac{|B(u_0,4r)|^{{\kappa}/p}}{|B(u_0,r)|^{1/q+\beta/Q}}\left(1+\frac{4r}{\rho(u_0)}\right)^{\theta}.
\end{split}
\end{equation*}
Using the inequalities \eqref{homonorm} and \eqref{2rx}, and noting the fact that $\beta/Q=\kappa/p-1/q$, we derive
\begin{equation*}
\begin{split}
K_1&\leq C_n\big\|f\big\|_{L^{p,\kappa}_{\rho,\theta}(\mathbb H^n)}\left(1+\frac{4r}{\rho(u_0)}\right)^{\theta}\\
&\leq C_{n,\theta}\big\|f\big\|_{L^{p,\kappa}_{\rho,\theta}(\mathbb H^n)}\left(1+\frac{r}{\rho(u_0)}\right)^{\theta}.
\end{split}
\end{equation*}
Let us now turn to estimate the term $K_2$. For any $u\in B(u_0,r)$,
\begin{equation*}
\begin{split}
\big|\mathcal I_{\alpha}f_2(u)-(\mathcal I_{\alpha}f_2)_B\big|
&=\bigg|\frac{1}{|B|}\int_B\big[\mathcal I_{\alpha}f_2(u)-\mathcal I_{\alpha}f_2(v)\big]\,dv\bigg|\\
&=\bigg|\frac{1}{|B|}\int_B\bigg\{\int_{(4B)^c}\Big[\mathcal K_{\alpha}(u,w)-\mathcal K_{\alpha}(v,w)\Big]f(w)\,dw\bigg\}dv\bigg|\\
&\leq\frac{1}{|B|}\int_B\bigg\{\int_{(4B)^c}\big|\mathcal K_{\alpha}(u,w)-\mathcal K_{\alpha}(v,w)\big|\cdot|f(w)|\,dw\bigg\}dv.
\end{split}
\end{equation*}
By using the same arguments as that of Theorem \ref{mainthm:1}, we find that
\begin{equation*}
|v^{-1}u|\leq |w^{-1}u|/2 \quad \& \quad |w^{-1}u|\approx |w^{-1}u_0|,
\end{equation*}
whenever $u,v\in B$ and $w\in(4B)^c$. This fact along with Lemma \ref{kernel2} yields
\begin{align}\label{average}
&\big|\mathcal I_{\alpha}f_2(u)-(\mathcal I_{\alpha}f_2)_B\big|\notag\\
&\leq\frac{C_{N,\alpha}}{|B|}\int_B\bigg\{\int_{(4B)^c}\bigg(1+\frac{|w^{-1}u|}{\rho(u)}\bigg)^{-N}
\frac{|v^{-1}u|^{\delta}}{|w^{-1}u|^{Q-\alpha+\delta}}\cdot|f(w)|\,dw\bigg\}dv\notag\\
&\leq C_{N,\alpha,n}\int_{(4B)^c}\bigg(1+\frac{|w^{-1}u_0|}{\rho(u)}\bigg)^{-N}\frac{r^{\delta}}{|w^{-1}u_0|^{Q-\alpha+\delta}}\cdot|f(w)|\,dw\notag\\
&=C_{N,\alpha,n}\sum_{k=2}^\infty\int_{2^kr\leq|w^{-1}u_0|<2^{k+1}r}
\bigg(1+\frac{|w^{-1}u_0|}{\rho(u)}\bigg)^{-N}\frac{r^{\delta}}{|w^{-1}u_0|^{Q-\alpha+\delta}}\cdot|f(w)|\,dw\notag\\
&\leq C_{N,\alpha,n}\sum_{k=2}^\infty\frac{1}{2^{k\delta}}\cdot\frac{1}{|B(u_0,2^{k+1}r)|^{1-({\alpha}/Q)}}
\int_{B(u_0,2^{k+1}r)}\bigg(1+\frac{2^kr}{\rho(u)}\bigg)^{-N}|f(w)|\,dw.
\end{align}
Furthermore, by using H\"older's inequality and \eqref{com2}, we deduce that for any $u\in B(u_0,r)$,
\begin{align}\label{end1.3}
&\big|\mathcal I_{\alpha}f_2(u)-(\mathcal I_{\alpha}f_2)_B\big|\notag\\
&\leq C\sum_{k=2}^\infty\frac{1}{2^{k\delta}}\cdot\frac{1}{|B(u_0,2^{k+1}r)|^{1-({\alpha}/Q)}}
\left(1+\frac{r}{\rho(u_0)}\right)^{N\cdot\frac{N_0}{N_0+1}}
\left(1+\frac{2^{k+1}r}{\rho(u_0)}\right)^{-N}\notag\\
&\times\bigg(\int_{B(u_0,2^{k+1}r)}\big|f(w)\big|^p\,dw\bigg)^{1/p}
\left(\int_{B(u_0,2^{k+1}r)}1\,dw\right)^{1/{p'}}\notag\\
&\leq C\big\|f\big\|_{L^{p,\kappa}_{\rho,\theta}(\mathbb H^n)}
\sum_{k=2}^\infty\frac{1}{2^{k\delta}}\cdot\left(1+\frac{r}{\rho(u_0)}\right)^{N\cdot\frac{N_0}{N_0+1}}
\left(1+\frac{2^{k+1}r}{\rho(u_0)}\right)^{-N}\notag\\
&\times\frac{|B(u_0,2^{k+1}r)|^{{\kappa}/p}}{|B(u_0,2^{k+1}r)|^{1/q}}\left(1+\frac{2^{k+1}r}{\rho(u_0)}\right)^{\theta}\notag\\
&=C\big\|f\big\|_{L^{p,\kappa}_{\rho,\theta}(\mathbb H^n)}
\sum_{k=2}^\infty\frac{|B(u_0,2^{k+1}r)|^{\beta/Q}}{2^{k\delta}}\cdot\left(1+\frac{r}{\rho(u_0)}\right)^{N\cdot\frac{N_0}{N_0+1}}
\left(1+\frac{2^{k+1}r}{\rho(u_0)}\right)^{-N+\theta},
\end{align}
where the last equality is due to the assumption $\beta/Q=\kappa/p-1/q$. From the pointwise estimate \eqref{end1.3} and \eqref{homonorm}, it readily follows that
\begin{equation*}
\begin{split}
K_2&=\frac{1}{|B|^{1+\beta/Q}}\int_B\big|\mathcal I_{\alpha}f_2(u)-(\mathcal I_{\alpha}f_2)_B\big|\,du\\
&\leq C\big\|f\big\|_{L^{p,\kappa}_{\rho,\theta}(\mathbb H^n)}
\sum_{k=2}^\infty\frac{1}{2^{k\delta}}\cdot\left(\frac{|B(u_0,2^{k+1}r)|}{|B(u_0,r)|}\right)^{\beta/Q}
\left(1+\frac{r}{\rho(u_0)}\right)^{N\cdot\frac{N_0}{N_0+1}}
\left(1+\frac{2^{k+1}r}{\rho(u_0)}\right)^{-N+\theta}\\
&\leq C\big\|f\big\|_{L^{p,\kappa}_{\rho,\theta}(\mathbb H^n)}
\sum_{k=2}^\infty\frac{1}{2^{k(\delta-\beta)}}\cdot\left(1+\frac{r}{\rho(u_0)}\right)^{N\cdot\frac{N_0}{N_0+1}},
\end{split}
\end{equation*}
where $N>0$ is a sufficiently large number so that $N>\theta$. Also observe that $\beta<\delta\leq1$, and hence the last series is convergent. Therefore,
\begin{equation*}
K_2\leq C\big\|f\big\|_{L^{p,\kappa}_{\rho,\theta}(\mathbb H^n)}\left(1+\frac{r}{\rho(u_0)}\right)^{N\cdot\frac{N_0}{N_0+1}}.
\end{equation*}
Fix this $N$ and set $\vartheta=\max\big\{\theta,N\cdot\frac{N_0}{N_0+1}\big\}$. Finally, combining the above estimates for $K_1$ and $K_2$, the inequality \eqref{end1.1} is proved and then the proof of Theorem \ref{mainthm:3} is finished.
\end{proof}

In the end of this article, we discuss the corresponding estimates of the fractional integral operator $I_{\alpha}=(-\Delta_{\mathbb H^n})^{-\alpha/2}$ (under $0<\alpha<Q$). We denote by $K^{*}_{\alpha}(u,v)$ the kernel of $I_{\alpha}=(-\Delta_{\mathbb H^n})^{-\alpha/2}$. In \eqref{claim}, we have already shown that
\begin{equation}\label{WH3}
\big|K^{*}_{\alpha}(u,v)\big|\leq C_{\alpha,n}\cdot\frac{1}{|v^{-1}u|^{Q-\alpha}}.
\end{equation}
Using the same methods and steps as we deal with \eqref{WH2} in Lemma \ref{kernel2}, we can also show that for some fixed $0<\delta\leq 1$ and $0<\alpha<Q$, there exists a positive constant $C_{\alpha,n}>0$ such that for all $u,v$ and $w$ in $\mathbb H^n$, 
\begin{equation}\label{WH4}
\big|K^{*}_{\alpha}(u,w)-K^{*}_{\alpha}(v,w)\big|\leq C_{\alpha,n}\cdot\frac{|v^{-1}u|^{\delta}}{|w^{-1}u|^{Q-\alpha+\delta}},
\end{equation}
whenever $|v^{-1}u|\leq |w^{-1}u|/2$. Following along the lines of the proof of Theorems \ref{mainthm:1}--\ref{mainthm:3} and using the inequalities \eqref{WH3} and \eqref{WH4}, we can obtain the following estimates of $I_{\alpha}$ with $\alpha\in(0,Q)$.
\begin{thm}\label{thm:1}
Let $0<\alpha<Q$, $1<p<Q/{\alpha}$ and $1/q=1/p-{\alpha}/Q$. If $0<\kappa<p/q$, then the fractional integral operator $I_{\alpha}$ is bounded from $L^{p,\kappa}(\mathbb H^n)$ into $L^{q,{(\kappa q)}/p}(\mathbb H^n)$.
\end{thm}

\begin{thm}\label{thm:2}
Let $0<\alpha<Q$, $p=1$ and $q=Q/{(Q-\alpha)}$. If $0<\kappa<1/q$, then the fractional integral operator $I_{\alpha}$ is bounded from $L^{1,\kappa}(\mathbb H^n)$ into $WL^{q,(\kappa q)}(\mathbb H^n)$.
\end{thm}
Here, we remark that Theorems \ref{thm:1} and \ref{thm:2} have been proved by Guliyev et al.\cite{guliyev}.
\begin{thm}\label{thm:3}
Let $0<\alpha<Q$, $1<p<Q/{\alpha}$ and $1/q=1/p-{\alpha}/Q$. If $p/q\leq\kappa<1$, then the fractional integral operator $I_{\alpha}$ is bounded from $L^{p,\kappa}(\mathbb H^n)$ into $\mathcal{C}^{\beta}(\mathbb H^n)$ with $\beta/Q=\kappa/p-1/q$ and $\beta<\delta\leq1$, where $\delta$ is given as in \eqref{WH4}.
\end{thm}
As an immediate consequence we have the following corollary.
\begin{cor}
Let $0<\alpha<Q$, $1<p<Q/{\alpha}$ and $1/q=1/p-{\alpha}/Q$. If $\kappa=p/q$, then the fractional integral operator $I_{\alpha}$ is bounded from $L^{p,\kappa}(\mathbb H^n)$ into $\mathrm{BMO}(\mathbb H^n)$.
\end{cor}
Upon taking $\alpha=1$ in Theorem \ref{thm:3}, we get the following \textbf{Morrey's lemma} on the Heisenberg group. 
\begin{cor} 
Let $\alpha=1$, $1<p<Q$ and $1/q=1/p-1/Q$. If $p/q<\kappa<1$, then the fractional integral operator $I_{1}$ is bounded from $L^{p,\kappa}(\mathbb H^n)$ into $\mathcal{C}^{\beta}(\mathbb H^n)$ with $\beta/Q=\kappa/p-1/q$ and $\beta<\delta\leq1$, where $\delta$ is given as in \eqref{WH4}. Namely,
\begin{equation*}
\big\|\nabla_{\mathbb H^n}f\big\|_{\mathcal{C}^{\beta}(\mathbb H^n)}\leq C\big\|f\big\|_{L^{p,\kappa}(\mathbb H^n)},
\end{equation*}
where $0<\kappa<1$, $p>(1-\kappa)Q$, $\beta=1-{(1-\kappa)Q}/p$ and the gradient $\nabla_{\mathbb H^n}$ is defined by
\begin{equation*}
\nabla_{\mathbb H^n}=\big(X_1,\dots,X_n,Y_1,\dots,Y_n\big).
\end{equation*}
\end{cor}

\end{document}